\documentclass[12pt,a4paper]{amsart}
\usepackage{amssymb}
\usepackage{mathrsfs}
\usepackage{amsfonts}
\usepackage[active]{srcltx}
\usepackage{enumerate}
\usepackage{color}

\usepackage{amsmath,amssymb,xspace,amsthm}

\newtheorem{theorem}{Theorem}[section]
\newtheorem{lemma}[theorem]{Lemma}
\newtheorem{definition}{Definition}[section]
\numberwithin{equation}{section}

\newcommand{\ad}{\operatorname{ad}\xspace}

\newcommand{\C}{\ensuremath{\mathbb C}\xspace}

\renewcommand{\a}{\ensuremath{\alpha}}
\renewcommand{\b}{\ensuremath{\beta}}

\renewcommand{\l}{\ensuremath{\lambda}}

\newcommand{\h}{\ensuremath{\mathfrak{h}}}

\newcommand{\Z}{\ensuremath{\mathbb{Z}}\xspace}
\newcommand{\N}{\ensuremath{\mathbb{N}}\xspace}

\newcommand{\W}{\ensuremath{\mathcal{W}}\xspace}

\renewcommand{\geq}{\geqslant}

\newcommand{\p}{\partial}

\def\BB{\mathcal{B}}

\def\fg{\mathfrak{g}}
\def\fh{\mathfrak{h}}
\def\sl{\mathfrak{sl}}
\def\gl{\mathfrak{gl}}

\newcommand{\bl}{\lambda\hskip-6pt\lambda}
\newcommand{\bmu}{\mu\hskip-6.5pt\mu}
\newcommand{\bm}{{\bf m}}
\newcommand{\bn}{{\bf n}}

\newcommand{\bo}{{\bf 0}}

\newcommand{\bp}{\partial\hskip-6.5pt\partial}
\begin{document}

\title[$U(\fh)$-free modules over the Block algebra $\BB(q)$]
{$U(\fh)$-free modules over the Block algebra $\BB(q)$}

\author{Xuewen Liu and Xiangqian Guo}

\date{} \maketitle

\begin{abstract}
In this paper, we construct a new class of modules over the Block algebra $\BB(q)$, 
where $q$ is a nonzero complex number. We determined the irreducibilities of these modules
and the isomorphisms among them.
\end{abstract}

\vskip 5pt \noindent{\em Keywords:} Block algebra, irreducible modules, $U(\fh)$-free modules.
\vskip 5pt

\section{Introduction}

Classification is one of the main objects in many areas of mathematical studies.
In Lie algebra theory, it is natural to classify certain Lie algebras with good properties
or to classify certain modules over a particular or a family of Lie algebras. 
The most famous works on the classification on complex Lie algebras are the classifications of simple finite-dimensional Lie algebras, Kac-Moody algebras, simple finite-dimensional Lie superalgebras, 
simple Lie algebras with finite growth and so on.
As to the classifications on modules of Lie algebras, there are also many important results, 
such as the classifications of simple Harish-Chandra modules (weight modules with finite-dimensional weight spaces) for various Lie algebras including the finite-diemnsional Lie algebras (\cite{Mat2}), Virasoro algebras (\cite{Mat3}), generalized Virasoro algebras (\cite{S1, LZ1, GLZ1}), Witt algebras (\cite{BF}), 
(generalized) Heisengerg-Virasoro algebras (\cite{LZ2, LG}), loop Virasoro algebra (\cite{GLZ2}), Lie algebras of Block type (\cite{S2, S3, SXX, CGZ}) and so on. 

It is natural to consider the problem of classifications of all simple weight modules and even all simple modules. However, this problem is so hard that people do not believe this aim can be achieved for any nontrivial Lie algera
until Block obtained the classification of simple modules for the $3$-dimensional Lie algebra $\sl_2$ and some related Lie algebras in \cite{Bl} in 1979.
Now, due to the knowledge of the authors, there are only very few Lie algebras whose simple modules are be classified (c.f. \cite{MZ1}), depending on Block's result. 
As to the problem for simple weight modules, the situation may be a little better. 
But there are still not many results, and the only examples besides those in \cite{Bl} and \cite{MZ1}
is the so called spatial aging algebra (see \cite{LMZ}).

Since it is almost impossible to classify all simple modules for general Lie algebras, 
one turns to consider similar problems for certain families of simple modules with good conditions.
For example Whittaker modules were constructed for 
finite dimensional simple Lie algebras (\cite{Ko}), affine Lie algebras (\cite{ALZ}), 
quantum groups (\cite{BO, O}), 
(generlized) Virasoro algebras (\cite{OW, Y, GL1, GLZ2}) and some other algebras 
(see \cite{BM, GL2, LiG, LWZ} and references therein),
and in most cases the simple Whittaker modules are classified. 
Other authors classified another classes of simple modules on which the action of positive part is locally nilpotent or locally finite for some Lie algebras including 
Virasoro algebras (\cite{MZ2}), Heisenberg-Virasoro algebra (\cite{MZ2, CG}), affine algebras (\cite{MZ2, GZ1, GZ2}).  

Recently, another class of modules, which many authors call ``$U(\fh)$-free modules", 
were introduced and studied.
Some classification results and other results are given for such modules over various Lie algebras.
Let us first recall some notations.

Throughout this paper, all algebras and modules are considered over the complex field, 
though many results still hold for arbitrary fields of characteristic $0$.
Let $\fg$ be a Lie algebra with Cartan subalgebra $\fh$, that is,
$\fh$ is a maximal abelian subalgebra and acts on $\fg$ semisimply. 
A $\fg$-module is called a weight module if $\fh$ acts on the module semisimply.
If we replace the condition of semi-simplicity by an opposite condition that $\fh$ acts freely on the module,
we get the $U(\fh)$-modules. 
More precisely, let $M$ be a $\fg$-module such that $M$ is a $U(\fh)$-free module of rank-$r$
via the restriction action of $U(\fh)$,
then we call $M$ a $U(\fh)$-free $\fg$-module of rank $r$. Note that $r$ may be a nonnegative integer or 
$+\infty$.

The notation of $U(\fh)$-free modules was first introduced by \cite{N1} for the simple Lie algebra
$\sl_n$. In this paper and a subsequent paper \cite{N2},
Nilsson showed that a finite dimensional simple Lie algebra has nontrivial $U(\fh)$-free modules 
if and only if it is of type $A$ or $C$. Moreover, he also gave a classification of $U(\fh)$-free modules of rank $1$ in these cases.

Almost at the same time, the $U(\fh)$-free $\sl_n$-modules were also studied in \cite{TZ1}, 
from a different point of view. 
They constructed a large class of $U(\fh)$-free $\sl_{n+1}$-modules from certain modules 
over Witt algebras $\W_{n}$, without mentioning the ``$U(\fh)$-free" feature.
The above mentioned $\W_n$-modules are also $U(\fh)$-free modules and
were first constructed from Weyl algebra modules in \cite{LZ3}.
For the special case $\W_1$ (centerless Virasoro algebra), 
the $U(\fh)$-free modules first appeared in an earlier paper \cite{GLZ3} as quotients of fraction modules.

Recently, many authors investigated the $U(\fh)$-free modules for various other Lie algebras.
In 2016, the $U(\fh)$-free modules of rank-$1$ for the Kac-Moody Lie algebras are determined in \cite{CTZ}.
In fact, there are no nontrivial $U(\fh)$-free modules of rank-$1$ for affine type and indefinite type,
so only the finite type algebras are left, which has been solved in Nilsson's paper.
Moreover, $U(\fh)$-free modules of rank-$1$ are also classified for the Witt algebras $\W_n$, 
simple finite-dimensional Lie superalgebras, Cartan type Lie algebras, Heisenberg-Virasoro algebra, W algebra $\W(2,2)$,
the algebras $\mathcal{V}ir(a,b)$ and so on (see \cite{TZ2, CZ, CG2, HCS, Z}).

There are also some results on $U(\fh)$-free modules of higher rank.
For example, a family of $U(\fh)$-free $\sl_2$-modules of arbitrary rank were constructed in \cite{MP},
while $U(\fh)$-free modules of higher rank for Witt algebras $\W_n$ are studied and described using 
Weyl algebra modules in \cite{LLZ, GLLZ}. 
In \cite{CG3}, $U(\fh)$-modules of infinite rank for Virasoro algebra were constructed from 
rank-$1$ $U(\fh)$-modules of the W algebra $\W(2,2)$. 
It is worthwhile to mention that irreducibilities of the tensor products of such modules and 
the modules in \cite{MZ1} are determined in \cite{LGW1, LGW2}, 
which generalize the similar results in \cite{TZ3, TZ4} of the tensor products of $U(\fh)$-free Virasoro modules of rank-$1$ and the modules in \cite{MZ1}.

In the present paper, we will consider the $U(\fh)$-free modules for the Block algebra $\BB(q)$, $q\in\C$, see Definition \ref{def}.
%
%
The Block algebra $\BB(q)$ is simple for generic $q\in\C$. 
More precisely, if $2q\notin\Z$, then $\BB(q)$ itself is a simple Lie algebra and otherwise, the quotient algebra $[\BB(q),\BB(q)]/{\mathcal Z}(\BB(q))$ is simple, where $[\BB(q),\BB(q)]$ is a codimensional-$1$ subalgebra and ${\mathcal Z}(\BB(q))$ is the $1$-dimensional or trivial center of $\BB(q)$. 
These Lie algebras provided a large classes of infinite dimensional simple Lie algebras with quite simple structures, in fact, only the Virasoro algebra has simpler structure. 

The structure and representation theory for the Block algebras $\BB(q)$ or its subalgebra 
$\BB_+(q)$ (See Definition \ref{def}) are widely investigated. The derivations and biderivations of $\BB(q)$ are completely determined 
in \cite{DZ, LGZ}. Irreducible quasi-finite modules of $\BB_+(q)$ are classified in a sequences of papers
(See \cite{SXX,CGZ,WT}). Quasi-finite modules of the algebra $\BB(q)$ for some special $q$ are also studied 
(see, for example, \cite{SY, GGS} and references therein). 

In the present paper, we will construct the $U(\fh)$-free modules for the Block algebras $\BB(q)$, which provide a large class of new simple nongraded module for $\BB(q)$. 
This paper is organized as follows. In section 2, we introduce the basic notations,
recall some known results on Witt algebra $\W_1$ and construct the $U(\fh)$-modules over $\BB(q)$.
We also determine the irreducibility of these modules and the isomorphisms among them. 
In fact, each such module has a unique nonzero submodule of codimension $1$.
In Section 3, we show that these modules exhaust all $U(\fh)$-free $\BB(q)$-modules of rank-$1$.

\section{Preliminaries results}\label{pre}

Denote by $\Z$, $\N$, $\Z_+$ and $\C$ the sets of integers, positive
integers, nonnegative integers and complex numbers respectively. 
Throughout this paper, all algebras, modules and vector spaces are over $\C$.
Let us first recall the definition of the Block Lie algebras.

\begin{definition}\label{def} Let $q\in\C$. 
The Block Lie algebra $\BB(q)$ is the Lie algebra with a basis $\{L_{\bn}| \bn\in\Z^2\}$ subject to the following Lie brackets
\begin{equation*}
[L_{\bm},L_{\bn}]=(n_1(m_2+q)-m_1(n_2+q))L_{\bm+\bn}, \ \forall\ \bm,\bn\in\Z^2,
\end{equation*}
here and after we always use $m_1, m_2$ to denote the entries of $\bm\in\Z^2$, i.e., 
$\bm=(m_1,m_2)$ and similar for $\bn$.
Denoted by $\BB_+(q)$ the subalgebra spanned by $L_{m,i}, m\in\Z, i\in\Z_+$.
\end{definition}

Note that the Lie algebras $\BB(q), q\neq0$, are in fact some special cases of generalized 
Block algebras studied in \cite{DZ}.
The Block algebra $\BB(q)$ is simple if $2q\not\in\Z$ and otherwise the quotient algebra
$[\BB(q),\BB(q)]/\C L_{0,-q}$ is simple. 
These algebras are among the infinite dimensional simple Lie algebras with simplest structures.

To define the $U(\fh)$-free modules for the algebra $\BB(q)$, we need first to extend $\BB(q)$ by adding 
the degree derivations
 $\p_1$ and $\p_2$  defined by
$$\p_1(L_{\bm})=m_1L_{\bm}, \  \  \p_2(L_{\bm})=m_2L_{\bm}, \forall\ \bm=(m_1,m_2)\in\Z^2.$$
Noticing $\ad L_\bo=q\p_1$, 
we denote $\BB'(q)=\BB(q)\ltimes(\C\p_1\oplus\C\p_2)$ if $q=0$ and $\BB'(q)=\BB(q)\ltimes\C\p_2$ if $q\neq0$. The Lie bracket of $\BB'(q)$ is just 
$$[\p+L_{\bm},\p'+L_{\bn}]=\p(L_{\bn})-\p'(L_{\bm})+[L_\bm,L_{\bn}],\ \forall\ \bm,\bn\in\Z^2, 
\p, \p'\in\C\p_1\oplus\C\p_2.$$

Obviously, $\fh=\C\p_1\oplus\C\p_2$ is the Cartan subalgebra of $\BB'(q)$. 
A $\BB'(q)$-module $M$ is called a 
$U(\fh)$-free module of rank $r$ if $M$ is a free $U(\fh)$-module of rank $r$ via restriction action; in this case, we also call $M$ a $U(\fh)$-free $\BB(q)$-module of rank $r$ by abuse of language.
When $q=0$, the algebra $\BB'(0)$ is the direct sum of the Virasoro-like algebra (also known as the Lie algebra of divergence zero vector fields on a $2$-dimensional torus) and a $1$-dimensional center $\C L_{\bo}$. 
The $U(\fh)$-free $\BB'(0)$-modules of rank $1$ are already characterized in a recent paper \cite{Z}, 
where $U(\fh)$-free modules of rank $1$ over Cartan type algebras are studied. So we always assume $q\neq0$ in the rest of the present paper.

The algebras $\BB(q)$ and $\BB'(q)$ have natural subalgebras isomorphic to the (rank-$1$) Witt algebra, 
also known as the centerless Virasoro algebra. 
The Witt algebra $\W=\W_1$ is the derivation Lie algebra of the Laurent polynomial algebra in one variable. More precisely, $\W=\bigoplus_{i\in\Z}\C d_i$ as a vector space and its Lie bracket is given by
\begin{equation*}
[d_{i},d_{j}]=(j-i)d_{i+j}, \ \forall\ i,j\in\Z.
\end{equation*}

By straightforward calculations, we can show
\begin{lemma}\label{witt} For any $\bm=(m_1,m_2)\in\Z^2$ with $m_1\neq0$, denote 
$d_{i\bm}=\frac{1}{m_1q}L_{i\bm}$ for all $i\in\Z$.
Then $\W_\bm=\bigoplus_{i\in\Z}\C d_{i\bm}\subseteq\BB(q)$ form a subalgebra isomorphic to the 
Witt algebra $\W$.
\end{lemma}
%

The theory of Harish-Chandra modules (weight modules with finite-dimensional weight spaces)
over Virasoro algebra (and hence over Witt algebra) 
are fairly well developed. 
Recently, some authors constructed many classes of simple non-Harish-Chandra modules, 
including simple weight modules with infinite-dimensional weight spaces (see
\cite{CGZ, CM, LLZ, LZ2, Z}) and simple non-weight modules (see
\cite{BM, LGZ, LLZ, LZ1, MW, MZ1, OW, TZ4, TZ1, TZ2, TZ3}).
Among these modules a class of non-weight modules play an important role
because of its close relations with many other weight and non-weight modules.
We first define such modules.

For any complex number $\l,\a\in\C$ with $\l\neq0$, set $\Omega(\l,\a)=\C[t]$, 
the polynomial space in the variable $t$, and define the action of $\W$ as 
\begin{equation}\label{Omega action}
d_if(t)=\l^i(t-i\a)f(t-i),\ \forall\ i\in\Z, f(t)\in\Omega(\l,\a).
\end{equation}


It is well known that
$\Omega(\lambda,\a)$ is simple if and only if $\a\in \C^*$ and that
if $\a=0,$ then $\Omega(\lambda,0)$ has a unique nonzero proper
submodule $t\Omega(\lambda,0)$, which has codimension $1$ (see
Theorem 13 of \cite{GLZ3} or Section 4.3 of \cite{LZ3}).

\begin{theorem}{\textrm{([TZ2])}}\label{omega} 
Any $\W$-module, which is free of rank $1$ when restricted to $U(\C d_0)$, is isomorphic to some
$\Omega(\l,\a)$, where $\l,\a\in\Z$ and $\l\neq0$. 
\end{theorem}

Similarly to the Witt algebra, we may define Omega module $\Omega(\bl,\a)$ over $\BB(q)$ and $\BB'(q)$ 
for any $\bl=(\l_1,\l_2)\in\C^2$ with $\l_1\l_2\neq0$ and $\a\in\C$. 
As a vector space $\Omega(\bl,\a)=\C[\p_1,\p_2]$, the polynomial space in two variables $\p_1,\p_2$,
the action of $\p_1, \p_2$ are the usual multiplication, and the $\BB'(q)$-actions of other elements 
are defined as
$$L_{m_1,m_2}f(\p_1,\p_2)=\l_1^{m_1}\l_2^{m_2}(\p_1(q\a+m_2)-m_1(\a+\p_2))f(\p_1-m_1,\p_2-m_2).$$

It is easy to see that the module $\Omega(\bl,\a)$ is a $U(\fh)$-free $\BB(q)$-module as well as $\BB'(q)$-modules of rank $1$, that is,
as a $U(\fh)$-module, $\Omega(\bl,\a)$ is a free module of rank $1$. In the next section, we will
show that these modules exhaust all modules with this property.
Here we first present some properties of these modules.

For any module $\Omega(\bl,\a)$, where $\bl=(\l_1,\l_2)\in\C^2$ with $\l_1\l_2\neq0$ and $\a\in\C$,
denote $\Omega'(\bl,\a)=\p_1\Omega(\bl,\a)+(q\a+\p_2)\Omega(\bl,\a)$. It is easy to see that
$\Omega'(\bl,\a)$ is a $\BB(q)$-submodule as well as a $\BB'(q)$-submodule of $\Omega(\bl,\a)$. The following theorem shows that

\begin{theorem}\label{simple} 
Suppose $\bl=(\l_1,\l_2)\in\C^2$ with $\l_1\l_2\neq0$ and $\a\in\C$, then
$\Omega'(\bl,\a)$ is the unique nonzero proper submodule of $\Omega(\bl,\a)$.
In particular, $\Omega'(\bl,\a)$ is a simple $\BB(q)$-module.
\end{theorem}

\begin{proof} 
By replacing $L_\bm$ with $L_\bm/\bl^\bm$, we may assume that $\bl=(1,1)$.
Let $W$ be a nonzero submodule of $\Omega(\bl,\a)$ and take any nonzero element $f(\bp)=\sum_{k=0}^m\sum_{l=0}^na_{k,l}\p_1^k\p_2^l\in W$.

Applying $L_{-m_1,-m_2}$ on $f(\bp)$, we see that 
$$\aligned
 &L_{-m_1,-m_2}f(\p_1,\p_2)\\
=&(q(\p_1+m_1\a)+(m_1\p_2-m_2\p_1))\sum_{k=0}^m\sum_{l=0}^na_{k,l}(\p_1+m_1)^k(\p_2+m_2)^l\\
=&(\p_1(q-m_2)+m_1(q\a+\p_2))\sum_{k=0}^m\sum_{l=0}^n\sum_{i=0}^k\sum_{j=0}^l{k\choose i}{l\choose j}a_{k,l}m_1^im_2^j\p_1^{k-i}\p_2^{l-j}\\
\endaligned$$ lies in $W$ for any $\bm=(m_1,m_2)\in\Z^2$.
Regarding this element as a polynomial in $m_1, m_2$ with
coefficients in $\Omega(\bl,\a)$,
and using the Vandermonde determinant we can deduce that each coefficient of $m_1^im_2^j$ of the elements lies in $W$. In particular, the coefficients of $m_1^{k+1}m_2^l$ and $m_1^km_2^{l+1}$ give
rise to $q\a+\p_2, \p_1\in W$. Hence $\Omega'(\bl,\a)\subseteq W$. 
Noticing that $\Omega(\bl,\a)/\Omega'(\bl,\a)$ is $1$-dimensional, we see that $W=\Omega(\bl,\a)$
or $W=\Omega'(\bl,\a).$
\end{proof}

\begin{theorem}\label{simple} 
Suppose $\bl=(\l_1,\l_2),\bmu=(\mu_1,\mu_2)\in\C^2$ with $\l_1,\l_2,\mu_1,\mu_2\neq0$ 
and $\a,\b\in\C$, then
$\Omega(\bl,\a)\cong \Omega(\bmu,\b)$ if and only $\bl=\bmu$ and $\a=\b$. 
Similarly, $\Omega'(\bl,\a)\cong \Omega'(\bmu,\b)$ if and only $\bl=\bmu$ and $\a=\b$.
\end{theorem}

\begin{proof} We only need to prove the necessity. 
Note thta $\Omega(\bl,\a)=\Omega(\bmu,\b)=\C[\p_1,\p_2]$ as vector spaces and denote the action
of $\BB(q)$ on $\Omega(\bl,\a)$ by $\cdot$ and the action on $\Omega(\bmu,\b)$ by $\circ$. 
Let $\varphi:\Omega(\bl,\a)\rightarrow\Omega(\bmu,\b)$ be a $\BB(q)$-module isomorphism.
It is obvious that 
$$U(\C\p_1\oplus\C\p_2)\circ\varphi(1)=\varphi(U(\C\p_1\oplus\C\p_2)\cdot 1)=\C[\p_1,\p_2],$$
which implies that $\phi(1)$ is a nonzero constant. Without loss of generality, we may assume that
$\varphi(1)=1$ and hence $\varphi(f(\bp))=f(\bp)$ for any polynomial $f(\bp)\in\C[\p_1,\p_2]$.
Now for any $\bm=(m_1,m_2)\in\Z^2$, we have 
$$\bl^\bm(\p_1(m_2+q)-m_1(\p_2+q\a)=\phi(L_{\bm}\cdot 1)=L_\bm\circ 1=\bmu^\bm(\p_1(m_2+q)-m_1(\p_2+q\b),$$
yielding $\bl=\bmu$ and $\a=\b$, as desired. The proof of the result for $\Omega'(\bl,\a)$ and $\Omega'(\bmu,\b)$
is similar and we omit the details.
\end{proof}

\section{$U(\fh)$-free modules over $\BB(q)$}\label{main}

In this section, we will determine the $U(\fh)$-free modules of rank $1$ over the algebra $\BB(q)$ or $\BB'(q)$, 
where $\fh=\C\p_1\oplus\C\p_2$ is the Cartan subalgebra of $\BB'(q)$. 

Before stating our main result, we first give some technique lemmas. 

\begin{lemma}\label{poly} Let $F(X,Y)$ be a polynomial in two variables $X,Y$.
If $F(X,Y)-F(X, Y-c)=aX+bY$ for some constant $a,b,c\in\C, c\neq0$, then
$F(X,Y)=(f(X)+(2aX+bc)Y+bY^2)/2c$ for some polynomial $f(X)$ in X.
\end{lemma}

\begin{proof} 
Suppose that $F(X,Y)=\sum_{i=0}^nf_i(X)Y^i$ for some polynomials $f_i(X)$ in $X$ and $n\in\Z_+$. 
It is clear that $f_i(X)=0$ for $i\geq 3$. So we may assume $n=2$. 
By the assumption of the lemma, we have
$$cf_1(X)+f_2(X)(2cY-c^2)=aX+bY,$$
forcing that $f_2(X)=b/2c$ and $f_1(X)=ac^{-1}X+b/2$, as desired.
\end{proof}

Let $M$ be a $\BB(q)$-module which is a free $U(\h)=\C[\p_1,\p_2]$-module of rank $1$ when restricted to $\fh$. Without loss of generality, we assume that $M=\C[\p_1,\p_2]$ and that the action of $\p_1$ and $\p_2$ are just left multiplications. To distinguish, we denote the module action by $\cdot$ and denote the polynomial multiplication by adjunction.

To be short, we denote $\bp=(\p_1,\p_2)$, $\bp-\bm=(\p_1-m_1, \p_2-m_2)$,
$\bm^\bot=(m_2,-m_1)$, $(\bm|\bn)=m_1n_1+m_2n_2$, $(\bp|\bm)=m_1\p_1+m_2\p_2$ for any $\bm=(m_1,m_2)$,
$\bn=(n_1,n_2)\in\Z^2$; in particular, we have $(\bp|\bm^\bot)=m_2\p_1-m_1\p_2$.

The following lemma is a direct consequence of the Lie bracket.

\begin{lemma}\label{f}
For any $f(\bp)=f(\p_1,\p_2)\in \C[\p_1,\p_2]$ and $\bm=(m_1,m_2)\in\Z^2$, we have
$$L_{\bm}\cdot f(\bp)=f(\bp-\bm)L_{\bm}\cdot 1.$$
\end{lemma}

Now to determine the module structure of $M$, it suffices to determine the action of
$L_{\bm}$ on $1\in M$ for any $\bm\in\Z^2$. Suppose that $L_{\bm}\cdot 1=g_{\bm}(\bp)$ for some polynomials $g_{\bm}\in\C[\p_1,\p_2]$. Note that $g_{\bo}(\bp)=q\p_1$ since $L_\bo=q\p_1$.

\begin{lemma}\label{g_m}
For any $\bm=(m_1,m_2)\in\Z^2$, we have $g_{\bm}=\l_1^{m_1}\l_2^{m_2}(\p_1(q+m_2)-m_1(q\a+\p_2))$ for some 
$\bl=(\l_1,\l_2)\in\C^2$ with $\l_1,\l_2\neq0$ and $\a\in\C$.
\end{lemma}

\begin{proof} For any $\bm=(m_1,m_2), \bn=(n_1,n_2)\in\Z^2$ and $f(\bp)\in M$, we have 
\begin{equation}\label{[L,L]}\begin{split}
 (L_\bm L_\bn-L_\bn L_\bm)\cdot f(\bp)
&=[L_\bm, L_\bn]\cdot f(\bp)\\
&=(n_1(m_2+q)-m_1(n_2+q))L_{\bm+\bn}\cdot f(\bp).
\end{split}\end{equation}
Taking $f(\bp)=1$ in \eqref{[L,L]}, we have by Lemma \ref{f} that
\begin{equation}\label{gg}\aligned
  g_\bn(\bp-\bm)
& g_\bm(\bp)-g_\bm(\bp-\bn)g_\bn(\bp)\\
& =(n_1(m_2+q)-m_1(n_2+q))g_{\bm+\bn}(\bp),\ \forall\ \bm,\bn\in\Z^2.
\endaligned\end{equation}
Take $\bn=-\bm$ in the above equation, then we have
$$\aligned
 g_\bm(\bp+\bm)g_{-\bm}(\bp)-g_{-\bm}(\bp-\bm)g_\bm(\bp)=2m_1q^2\p_1,
\endaligned$$
for all $\bm\in\Z^2$, where we have used the fact $g_{\bo}(\bp)=q\p_1$. Set 
\begin{equation}\label{G g}
G_\bm(\bp)= g_\bm(\bp+\bm)g_{-\bm}(\bp),
\end{equation} 
then we have
\begin{equation}\label{G}
G_\bm(\bp)-G_\bm(\bp-\bm)=2m_1q^2\p_1.
\end{equation}

For any $\bm\in\Z^2$, we set that $X_\bm(\bp)=(\bp|\bm^\bot)$.
Note that $X_{i\bm}=iX_\bm$ for any $i\in\Z$ and
$X_\bm(\bp-\bn)= X_\bm-(\bn|\bm^\bot)$ for all $\bn\in\Z^2$.

In case $m_1=0, m_2\neq0$, the equations \eqref{G} and \eqref{G g} imply that both $G_\bm(\bp)$ and $g_\bm(\bp)$ are independent on $\p_2$ and hence polynomials in $X_\bm=m_2\p_1$. We have $G_\bm(\bp)=F_\bm(X_\bm)$ and $g_\bm(\bp)=f_\bm(X_\bm)$ for some polynomials $F_\bm(X_\bm)$ and $f_\bm(X_\bm)$ in $\C[X_\bm]$

Now suppose $m_1\neq0$, then it is clear that $\C[\p_1,\p_2]=\C[X_\bm,\p_1]$ as polynomial rings. 
So we have $G_\bm(\bp)= F_\bm(X_\bm,\p_1)$ and $g_\bm(\bp)= f_\bm(X_\bm,\p_1)$ for some polynomials 
$F_\bm(X_\bm,\p_1)$ and $f_\bm(X_\bm,\p_1)$ in $\C[X_\bm,\p_1]$.
Then the equations \eqref{G g} and \eqref{G} can be rewritten as
\begin{equation}\label{F f}\aligned
   F_\bm(X_\bm,\p_1)&=f_\bm(X_\bm,\p_1+m_1)f_{-\bm}(X_{-\bm},\p_1).
\endaligned\end{equation}
and
\begin{equation}\label{F}
F_\bm(X_\bm,\p_1)-F_\bm(X_\bm,\p_1-m_1)=2m_1q^2\p_1.
\end{equation}
By Lemma \ref{poly}, we get 
\begin{equation}\label{F h}
F_\bm(X_\bm,\p_1)=h_\bm(X_\bm)+q^2\p_1(\p_1+m_1)
\end{equation}
for some polynomial $h_\bm(X_\bm)$ in $X_\bm$.

Note that the equations \eqref{F f}, \eqref{F} and \eqref{F h} hold for all $\bm\in\Z^2$
if we write $f_\bo(X_\bo,\p_1)=g_\bo(\bp)=q\p_1$ and $F_\bo(X_\bo,\p_1)=q^2\p_1^2$ in case $\bm=\bo$, 
and $f_\bm(X_\bm,\p_1)=f_\bm(X_\bm)$ and $F_\bm(X_\bm,\p_1)=F_\bm(X_\bm)$ in case $m_1=0, m_2\neq0$.
In the following, we will determine the polynomial $f_\bm(X_\bm,\p_1)$ in through several claims.

\noindent{\bf Claim 1.} $f_\bm(X_\bm,\p_1)=q\l_\bm(\p_1-m_1\a_\bm)+\l_\bm\phi_\bm(X_\bm,\p_1)$ 
for some $\l_\bm,\a_\bm\in\C$ with $\l_\bm\neq0$, $\l_{i\bm}=\l_\bm^i, \a_{i\bm}=\a_\bm$ for $i\in\Z\setminus\{0\}$, and some polynomial $\phi_\bm(X_\bm,\p_1)\in \C[X_\bm,\p_1]$
such that $\phi_\bm(0,\p_1)=0$ if $m_1\neq0$.

The claim for $\bm=\bo$ is trivial by taking $\bl_\bo=1$, $\a_\bo$ arbitrary and $\phi_\bo(X_\bo,\p_1)=0$, since $f_\bo(X_\bo,\p_1)=q\p_1$. 

If $m_1=0, m_2\neq0$, then claim is already true since $X_\bm=m_2\p_1$ and 
$f_\bm(X_\bm,\p_1)$ is a polynomial in $X_\bm$. 
We can choose $\l_{0,1}, \a_{0,1}\in\C$ with $\l_{0,1}\neq0$, 
set $\l_{0,m_2}=\l_{0,1}^{m_2}$, $\a_{0,m_2}=\a_{\bo}$ for all $m_2\in\Z$ and finally 
take $\phi_\bm(X_\bm,\p_1)=\l_{\bm}^{-1}f_\bm(X_\bm,\p_1)-q\p_1$ for $\bm=(0,m_2)$.
Note that the choice of $\l_{0,1}\neq0$ and $\a_{\bo}$ are arbitrary and we will choose 
appropriate values for them later so that the expression of $f_\bm(X_\bm,\p_1)$ may be compatible with that in the case $m_1\neq0$.

Now suppose $m_1\neq 0$. By Lemma \ref{witt}, we have the subalgebra $\W_\bm$ spanned by 
$d_{i\bm}=\frac{1}{qm_1}L_{i\bm}$, which is isomorphic to the classical Witt algebra $\W$. 
Note that $\p_1=q^{-1}L_\bo=m_1d_\bo$, $[X_\bm,d_{i\bm}]=0$ and $[\p_1, d_{i\bm}]=im_1d_{i\bm}$ for all $i\in\Z$. 
It is obvious that $X_\bm^kM$ is a $\W_\bm$-submodules of $M$ for each $k\in\Z_+$. 
Hence $X_\bm^kM/X_\bm^{k+1}M$ is a quotient $\W_\bm$-module and the action of $\p_1$ is given by
the multiplication. On other words, $X_\bm^kM/X_\bm^{k+1}M$ is a $U(d_\bo)$-free module of rank $1$.
In particular, by Lemma \ref{omega}, we have a $\W_\bm$-module isomorphism 
$$\varphi_\bm:\ \Omega(\l_\bm,\a_\bm)=\C[t]\rightarrow M/X_\bm M,$$
where $\l_\bm, \a_\bm\in\C$ with $\l_\bm\neq0$ 
and the $\W_\bm$-module action of $\Omega(\l_\bm,\a_\bm)$ is given as in \eqref{Omega action} via the Li algebra isomorphism in Lemma \ref{witt}. 
It is clear that $\varphi_\bm(1)$ is the coset of some nonzero constant element in $M=\C[X_\bm,\p_1]$. 
Without loss of generality, we may assume that $\varphi_\bm(1)=1\ \mod\ X_\bm M$. 
By the definition of $\Omega(\l_\bm,\a_\bm)$, we see that 
$d_{i\bm}\cdot 1=\l_\bm^i(d_\bo-i\a_\bm)\ \mod X_\bm M$, or equivalently,
$$L_{i\bm}\cdot 1=q\l_\bm^i(\p_1-im_1\a_\bm)\mod X_\bm M,\ \forall\ i\in\Z.$$
That is, 
\begin{equation}\label{f_im}
f_{i\bm}(X_{i\bm},\p_1)=q\l_\bm^i(\p_1-im_1\a_\bm)\mod X_\bm M,\ \forall\ i\in\Z.
\end{equation}

Replacing $\bm$ with $j\bm, j\neq0$ and $i$ with $1$, we have 
$f_{j\bm}(X_{j\bm},\p_1)=q\l_{j\bm}(\p_1-jm_1\a_{j\bm})\mod X_{j\bm} M$.
Recalling $X_{j\bm}=jX_{\bm}$ for $j\neq0$, we have
$f_{j\bm}(X_{j\bm},\p_1)=q\l_{j\bm}(\p_1-jm_1\a_{j\bm})\mod X_{\bm} M$,
and comparing 
with \eqref{f_im}, we can get $\l_{j\bm}=\l_\bm^j$ and $\a_{j\bm}=\a_\bm$ for all $j\neq0$ 
and $\bm\in\Z^2$ with $m_1\neq0$. Set $\phi_\bm(X_\bm,\p_1)=\bl_\bm^{-1}f_\bm(X_\bm,\p_1)-q(\p_1-m_1\a_{\bm})\in X_\bm\C[X_\bm,\p_1]$. 
The claim follows.

\noindent {\bf Claim 2.} $\phi_\bm(X_\bm, \p_1)$ is a polynomial in $X_\bm$.

This claim for $m_1=0$ is clear. 

Assume $m_1\neq0$ in the following proof of Claim 2.
Substitute the expression for $f_\bm(X_\bm,\p_1)$ in Claim 1 into \eqref{F h} and using \eqref{F f}, we obtain
$$\aligned
  & \big(q(\p_1+m_1-m_1\a_\bm)+\phi_\bm(X_\bm,\p_1+m_1)\big)\big(q(\p_1+m_1\a_\bm)
      +\phi_{-\bm}(X_{-\bm},\p_1)\big)\\
= & h_\bm(X_\bm)+q^2\p_1(\p_1+m_1). 
\endaligned$$
This implies that 
\begin{equation}\label{h phi}\aligned
h_\bm(X_\bm)=
 &q(\p_1+m_1-m_1\a_\bm)\phi_{-\bm}(X_{-\bm},\p_1)+q(\p_1+m_1\a_\bm)\phi_\bm(X_\bm,\p_1+m_1)\\
 &\hskip5pt +\phi_\bm(X_\bm,\p_1+m_1)\phi_{-\bm}(X_{-\bm},\p_1)+m_1^2q^2\a_\bm(1-\a_\bm).
\endaligned\end{equation}
Recalling that $X_{-\bm}=-X_\bm$, we can regard the above formula as an equation of
polynomials in $X_\bm$ and $\p_1$, which are algebraically independent since $m_1\neq0$.

Suppose on the contrary that $\phi_\bm(X_\bm,\p_1)\phi_{-\bm}(X_{-\bm},\p_1)\notin\C[X_\bm]$.
Let $X_\bm^a\p_1^b$ and $X_\bm^c\p_1^d$ be the highest degree terms of $\phi_\bm(X_{\bm},\p_1)$ and 
$\phi_{-\bm}(X_{-\bm},\p_1)$ respectively, with respect to the lexicographical order given by requiring that  the degree of $\p_1$ is greater than that of $X_\bm$. Then $X_\bm^{a+c}\p_1^{b+d}$, $X_\bm^c\p_1^{d+1}$ and $X_\bm^a\p_1^{b+1}$ are the highest degree terms of $\phi_\bm(X_\bm,\p_1+m_1)\phi_{-\bm}(X_{-\bm},\p_1)$, $(\p_1+m_1-m_1\a_\bm)\phi_{-\bm}(X_{-\bm},\p_1)$ and $(\p_1+m_1\a_\bm)\phi_\bm(X_\bm,\p_1+m_1)$ respectively. Noticing that $a, c\geq 1$, we see that if $b,d\geq1$, then $\deg(X_\bm^{a+c}\p_1^{b+d})>\deg(X_\bm^c\p_1^{d+1})$ and $\deg(X_\bm^{a+c}\p_1^{b+d})>\deg(X_\bm^a\p_1^{b+1})$, impossible; if $b=0, d\geq1$, then $\deg(X_\bm^c\p_1^{d+1})>\deg(X_\bm^a\p_1^{b+1})$ and 
$\deg(X_\bm^c\p_1^{d+1})>\deg(X_\bm^{a+c}\p_1^{b+d})$, again impossible; and similarly for the case $b\geq1, d=0$. So we must have $b=d=0$ and hence $\phi_\bm(X_\bm,\p_1)\in\C[X_\bm]$ for all $\bm\in\Z^2$ with $m_1\neq0$. This claim is completed.

By the above claim, we will write $\phi_\bm(X_\bm)$ instead of $\phi_\bm(X_\bm,\p_1)$ and 
$f_\bm(X_\bm,\p_1)=q\l_\bm(\p_1-m_1\a_\bm)+\l_\bm\phi_\bm(X_\bm)$ in what follows.
Note that $\bl_\bo=1$, $\phi_\bo(X_\bo)=\bo$ and $\phi_\bm(X_\bm)\in X_\bm\C[X_\bm]$ for $m_1\neq0$.

\noindent{\bf Claim 3.} $\phi_\bm(X_{\bm})$ is a polynomial in $X_\bm$ of degree $1$ for any $\bm\neq (0,-2q)$.

Considering the terms involving $\p_1$ in the equation \eqref{h phi}, we see 
$\phi_{-\bm}(X_{-\bm})=-\phi_\bm(X_\bm)$ for all $\bm\in\Z^2$ with $m_1\neq0$ 
and this equation becomes
\begin{equation*}\aligned
h_\bm(X_\bm)=&-(\phi_{\bm}(X_{\bm})-qm_1(\a_\bm-1))(\phi_\bm(X_\bm)-qm_1\a_\bm),\ {\text if}\ m_1\neq0.\\
\endaligned\end{equation*}

Now taking $\bm, \bn\in\Z^2$ and substituting $g_\bm(\bp)=f_\bm(X_\bm,\p_1)$ and $g_\bn(\bp)=f_\bn(X_\bn,\p_1)$ in \eqref{gg}, we get
\begin{equation}\label{ff}\aligned
  &f_\bn(X_\bn-(\bm|\bn^\bot),\p_1-m_1)
 f_\bm(X_\bm,\p_1)-f_\bm(X_\bm-(\bn|\bm^\bot),\p_1-n_1)f_\bn(X_\bn,\p_1)\\
=&(n_1(m_2+q)-m_1(n_2+q))f_{\bm+\bn}(X_{\bm+\bn},\p_1).
\endaligned\end{equation}
Substitute $f_\bm(X_\bm,\p_1)=q\l_\bm(\p_1-m_1\a_\bm)+\l_\bm\phi_\bm(X_\bm)$ into the above equation, 
we get
\begin{equation}\label{phiphi}\aligned
  &\l_\bm\l_\bn\Big(q(\p_1-m_1-n_1\a_\bn)+\phi_\bn(X_\bn-(\bm|\bn^\bot))\Big)\Big(q(\p_1-m_1\a_\bm)+\phi_\bm(X_\bm)\Big)\\
 &-\l_\bm\l_\bn\Big(q(\p_1-n_1-m_1\a_\bm)+\phi_\bm(X_\bm-(\bn|\bm^\bot))\Big)\Big(q(\p_1-n_1\a_\bn)+\phi_\bn(X_\bn)\Big)\\
=&(n_1(m_2+q)-m_1(n_2+q))\l_{\bm+\bn}\Big(q(\p_1-(m_1+n_1)\a_{\bm+\bn})+\phi_{\bm+\bn}(X_{\bm+\bn})\Big).
\endaligned\end{equation}
for all $\bm,\bn\in\Z^2$. 
Taking $\bm=(-m,0), \bn=(m,n)$ with $m, n\in\Z\setminus\{0\}$ in \eqref{phiphi}, and noticing 
$\l_{-m,0}=\l_{m,0}^{-1}$, $\a_{-m,0}=\a_{m,0}$ by Claim 1 and $\phi_{-m,0}(X_{-m,0})=\phi_{-m,0}(-X_{m,0})=-\phi_{m,0}(X_{m,0})$, we get 
\begin{equation}\label{-m0 mn}\aligned
  &\Big(q(\p_1+m-m\a_{m,n})+\phi_{m,n}(X_{m,n}+mn)\Big)\Big(q(\p_1+m\a_{m,0})-\phi_{m,0}(X_{m,0})\Big)\\
 &-\Big(q(\p_1-m+m\a_{m,0})-\phi_{m,0}(X_{m,0}+mn)\Big)\Big(q(\p_1-m\a_{m,n})+\phi_{m,n}(X_{m,n})\Big)\\
=&\l_{m,n}^{-1}\l_{m,0}\l_{0,n}(mq+m(n+q))\Big(q\p_1+\phi_{0,n}(X_{0,n})\Big).
\endaligned\end{equation}
Substituting $X_{m,n}=n\p_1-m\p_2, X_{m,0}=-m\p_2$ and $X_{0,n}=n\p_1$ into the above equation, 
we obtain an equation of polynomials in $\p_1, \p_2$. 
Denote by $\psi(\p_1,\p_2)$ the polynomial on the left hand side of \eqref{-m0 mn}, 
which is independent of $\p_2$ since the expression on the right hand side of \eqref{-m0 mn} 
is a polynomial in $\p_1$.

Suppose that $\phi_{m,0}(-m\p_2)=\sum_{i=1}^ka_i\p_2^i$ and 
$\phi_{m,n}(n\p_1-m\p_2)=\sum_{j=1}^lb_j(\p_2+\p'_1)^j$, where $\p'_1=-n\p_1/m$ and $k,l\geq1, a_kb_l\neq0$.
The we can write  
$$\psi(\p_1,\p_2)=\psi_1(\p_1,\p_2)+\psi_2(\p_1,\p_2)+\psi_3(\p_1),$$
where
\begin{equation*}\aligned
\psi_1(\p_1,\p_2)
=&\phi_{m,0}(-m(\p_2-n))\phi_{m,n}(n\p_1-m\p_2)-\phi_{m,0}(-m\p_2)\phi_{m,n}(n\p_1-m(\p_2-n))\\
=&\sum_{i=1}^ka_i(\p_2-n)^i\sum_{j=1}^lb_j(\p_2+\p'_1)^j-\sum_{i=1}^ka_i\p_2^i\sum_{j=1}^lb_j(\p_2-n+\p'_1)^j,\\
\psi_2(\p_1,\p_2)
=&\phi_{m,n}(n\p_1-m(\p_2-n))q(\p_1+m\a_{m,0})-q(\p_1+m-m\a_{m,n})\phi_{m,0}(-m\p_2)\\
 &+\phi_{m,0}(-m(\p_2-n))q(\p_1-m\a_{m,n})-q(\p_1-m+m\a_{m,0})\phi_{m,n}(n\p_1-m\p_2)\\
\endaligned\end{equation*}
are polynomials in $\p_1, \p_2$ and 
\begin{equation*}\aligned
&\psi_3(\p_1,\p_2)=q^2(\p_1+m-m\a_{m,n})(\p_1+m\a_{m,0})-q^2(\p_1-m+m\a_{m,0})(\p_1-m\a_{m,n})\\
\endaligned\end{equation*}
are polynomials in $\p_1$.

If $k,l\geq 2$, then the coefficients of $\p_2^{k+l-1}$ in $\psi_1(\p_1,\p_2)$ and $\psi_2(\p_1,\p_2)$ 
are $na_kb_l(l-k)$ and $0$ respectively, forcing $k=l$. 
Now, assuming $k=l\geq2$, we see the coefficients of $\p_1'\p_2^{2k-2}$ in $\psi_1(\p_1,\p_2)$ and 
$\psi_2(\p_1,\p_2)$ are $-nka_kb_k\neq0$ and $0$ respectively, contradiction.
Thus we have $k=1$ or $l=1$.

If $k=1, l>1$, 
the coefficients of $\p_2^{l}$ in $\psi_1(\p_1,\p_2)$ and $\psi_2(\p_1,\p_2)$ 
are $na_1b_l(l-1)$ and $qmb_l$ respectively. Hence $na_1(l-1)+qm=0$.
Applying the above arguments with $m,n$ replaced by $-m,-n$ respectively, and noticing that
$\phi_{-m,0}(m\p_2)=-\phi_{m,0}(-m\p_2)=-\sum_{i=1}^ka_i\p_2^i$ and 
$\phi_{-m,-n}(m\p_2-n\p_1)=-\phi_{m,n}(n\p_1-m\p_2)=-\phi_{m,n}(n\p_1-m\p_2)$, 
we have 
$na_1(l-1)-qm=0$, contradiction.

Similarly, if $l=1, k>1$, 
the coefficients of $\p_2^{k}$ in $\psi_1(\p_1,\p_2)$ and $\psi_2(\p_1,\p_2)$ 
are $na_kb_1(1-k)$ and $-qma_k$ respectively. Hence $nb_1(1-k)-qm=0$.
Applying the above arguments with $m,n$ replaced by $-m,-n$ respectively, 
we have $nb_1(1-k)+qm=0$, contradiction.

So we must have $k=l=1$, that is,
$\phi_{m,n}(X_{m,n})$ is a degree-$1$ polynomial in $X_{m,n}$ if $m\neq0$.
At last \eqref{-m0 mn} implies that $\phi_{0,n}(X_{0,n})$ is also a degree-$1$ polynomial in $X_{0,n}$
if $n\neq-2q$. The claim holds.

\noindent{\bf Claim 4.} $\phi_{\bm}=X_\bm$ and $f_\bm(X_\bm,\p_1)=\bl^\bm q(\p_1-m_1\a)+\bl^\bm X_\bm$, 
where $\a\in\C$, $\bl=(\l_1,\l_2)$ for some nonzero $\l_1,\l_2\in\C$ and $\bl^\bm=\l_1^{m_1}\l_2^{m_2}$.

Now we can assume that $\phi_{\bm}=a_\bm X_\bm+b_\bm$ for suitable $a_\bm, b_\bm\in\C$ for all $\bm\neq(0,-2q)$. Note that $b_\bm=0$  for $m_1\neq0$ or $\bm=\bo$ by Claim 1. 
Since $X_\bo=0$, we may assume $a_\bo=1$.

Substitute these value into \eqref{phiphi} for  
$\bm, \bn\in\Z^2\setminus\{0\}$, we get
\begin{equation*}\aligned
  &\l_\bm\l_\bn\Big(q(\p_1-m_1-n_1\a_\bn)+a_\bn(X_\bn-(\bm|\bn^\bot))+b_\bn\Big)\Big(q(\p_1-m_1\a_\bm)+a_\bm X_\bm+b_\bm\Big)\\
 &-\l_\bm\l_\bn\Big(q(\p_1-n_1-m_1\a_\bm)+a_\bm(X_\bm-(\bn|\bm^\bot))+b_\bm\Big)\Big(q(\p_1-n_1\a_\bn)+a_\bn X_\bn+b_\bn\Big)\\
=&(n_1(m_2+q)-m_1(n_2+q))\l_{\bm+\bn}\Big(q(\p_1-(m_1+n_1)\a_{\bm+\bn})+a_{\bm+\bn}X_{\bm+\bn}+b_{\bm+\bn}\Big).
\endaligned\end{equation*}
for all $\bm,\bn\in\Z^2\setminus\{\bo\}$. 
Considering the coefficients of $\p_1, \p_2$ and constant term, we deduce
\begin{equation}\label{p_1}\aligned
&\Big(q+a_\bn n_2\Big)\Big(qn_1+a_\bm(\bn|\bm^\bot)\Big)-\Big(q+a_\bm m_2\Big)\Big(qm_1+a_\bn(\bm|\bn^\bot)\Big)\\
=&\l_\bm^{-1}\l_\bn^{-1}\l_{\bm+\bn}(n_1(m_2+q)-m_1(n_2+q))\Big(q+a_{\bm+\bn}(m_2+n_2)\Big),
\endaligned\end{equation}
\begin{equation}\label{p_2}\aligned
&a_\bn n_1\Big(qn_1+a_\bm(\bn|\bm^\bot)\Big)-a_\bm m_1\Big(qm_1+a_\bn(\bm|\bn^\bot)\Big)\\
=&\l_\bm^{-1}\l_\bn^{-1}\l_{\bm+\bn}(n_1(m_2+q)-m_1(n_2+q))(m_1+n_1)a_{\bm+\bn},
\endaligned\end{equation}
\begin{equation}\label{constant}\aligned
  &\Big(qn_1+a_\bm(\bn|\bm^\bot)\Big)\Big(qn_1\a_\bn-b_\bn\Big)-\Big(qm_1+a_\bn(\bm|\bn^\bot)\Big)\Big(qm_1\a_\bm-b_\bm\Big)\\
=&\l_\bm^{-1}\l_\bn^{-1}\l_{\bm+\bn}(n_1(m_2+q)-m_1(n_2+q))\Big(q(m_1+n_1)\a_{\bm+\bn}-b_{\bm+\bn}\Big).
\endaligned\end{equation}
Taking $n_1=-m_1\neq0$ in \eqref{p_2}, we have $a_\bm=a_\bn$ for all $\bm,\bn\in\Z^2$ with $m_1=-n_1\neq0$ and hence for all $\bm,\bn\in\Z^2$ with $m_1=n_1\neq0$.
Then taking $n_1=m_1\neq0, m_2\neq n_2$ in \eqref{p_2} and \eqref{p_1}, we can deduce 
$\l_\bm\l_\bn a^2_\bm=\l_{\bm+\bn}a_{\bm+\bn}$ and $\l_\bm\l_\bn a_\bm=\l_{\bm+\bn}$ respectively. 
Replacing $\bm,\bn$ with $2\bm, 2\bn$ respectively and noticing $\l_{2\bm}=\l_\bm^2,\l_{2\bn}=\l_\bn^2$, 
$\l_{2(\bm+\bn)}=\l_{\bm+\bn}^2$ and $a_{2\bm}=a_{\bm}$, 
we have $\l_{\bm}^2\l_{\bn}^2 a_{\bm}=\l_{\bm+\bn}^2$ and hence $\l_\bm\l_\bn=\l_{\bm+\bn}$ and 
$a_\bm=1$ for all $\bm,\bn\in\Z^2$ with $n_1=m_1\neq0, m_2\neq n_2$.

Taking $m_1,n_1,m_1+n_1\neq0$ in \eqref{p_2}, we get $\l_\bm\l_\bn=\l_{\bm+\bn}$ if $n_1(m_2+q)-m_1(n_2+q)\neq0$.
Taking $n_1=0$, $n_2\neq-q，-2q$ and $m_1\neq0$ in \eqref{p_2}, we get
\begin{equation*}\aligned
\l_\bm \frac{\l_\bn(q+n_2a_\bn)}{q+n_2}=\l_{\bm+\bn}.
\endaligned\end{equation*}
Combining these two facts, we can find $\l_1,\l_2\in\C\setminus\{0\}$ such that
$\l_\bm=\l_1^{m_1}\l_2^{m_2}$ if $m_1\neq0$ and $\frac{\l_\bn(q+n_2a_\bn)}{q+n_2}=\l_2^{n_2}$ 
if $n_1=0$, $n_2\neq-q,-2q$. Thus we have
$$f(X_\bn)=\l_\bn(q\p_1+a_\bn X_\bn+b_\bn)=\l_2^{n_2}(q\p_1+X_\bn+\l_2^{-n}\l_\bn b_\bn),\ \text{if}\ n_1=0, n_2\neq-q,-2q.$$
So by adjusting the value of $\l_\bn, a_\bn$ and $b_\bn$, we may assume that
$\l_\bn=\l_2^{n_2}$, $a_\bn=1$ for $\bn\in\Z^2$ with $n_1=0, n_2\neq-q,-2q$.
Taking $n_1+m_1=0, m_2+n_2=-q$ in \eqref{p_1}, we deduce $a_{0,-q}=1$
and $f_{0,-q}(X_{0,-q},\p_1)=\bl_{0,-q}b_\bm$ if $q\in\Z$. 
Adjusting the values of $\bl_{0,-q}$ and $b_{0,-q}$, we may assume that $\bl_{0,-q}=\l_2^{-q}$.

Take $m_1,n_1,m_1+n_1\neq0$ in \eqref{constant}, we get
\begin{equation*}
n_1(qn_1+(\bn|\bm^\bot))\a_\bn-m_1(qm_1-(\bn|\bm^\bot))\a_\bm=
(n_1+m_1)\big(q(n_1-m_1)+(\bn|\bm^\bot)\big)\a_{\bm+\bn}.
\end{equation*}
Replacing $\bm, \bn$ with $\bm+\bn$ and $-\bn$ respectively and noticing $\a_\bn=\a_{-\bn}$, we get
\begin{equation*}
(m_1+n_1)(q(m_1+n_1)+(\bn|\bm^\bot))\a_{\bm+\bn}-n_1(qn_1+(\bn|\bm^\bot))\a_\bn=
m_1\big(q(m_1+2n_1)+(\bn|\bm^\bot)\big)\a_{\bm}.
\end{equation*}
Canceling $\a_{\bm+\bn}$ from the above two equations, we deduce $(qn_1+(\bn|\bm^\bot))(\a_\bm-\a_\bn)=0$ and hence $\a_\bm=\a_\bn$ for $\bm,\bn\in\Z^2$ with $m_1,n_1,m_1+n_1,qn_1+(\bn|\bm^\bot)\neq0$.
For any $\bm,\bn\in\Z^2$ with $m_1,n_1\neq0$, we may find some $i,j\in\N$ such that $im_1+jn_1\neq0$ and $qjn_1+(j\bn|i\bm^\bot)\neq0$ and hence $\a_\bm=\a_{i\bm}=\a_{j\bn}=\a_\bn$. 
Then by adjusting the value $\a_{\bn}$ with $n_1=0$, we may assume that $\a_\bm=\a_\bn$ for all
$\bm,\bn\in\Z^2\setminus\{\bo\}$, see the remark in the second paragraph of Claim 2. 
Denote this value by $\a\in\C$, we then have $\a_\bm=\a$ for all $\bm\in\Z^2\setminus\{\bo\}$. 

At last, we take $n_1=-m_1\neq0$ in \eqref{constant},
and deduce 
$(n_1(m_2+q)-m_1(n_2+q))b_{\bm+\bn}=0$, which implies $b_{\bn}=0$ for $\bn\in\Z^2$ with $\bn\neq(0,-2q)$.
Taking $\bm=(0,-2q)$ in \eqref{p_2} and \eqref{constant}, we can easily deduce $b_{0,-2q}=0$ and 
$\bl_{0,-2q}(2a_{0,-2q}-1)=\l_2^{-2q}$ in case $2q\in\Z$. Hence 
$f_{0,-2q}(X_{0,-2q},\p_1)=\bl_{0,-2q}(1-2a_{0,-2q})q\p_1=\l_2^{-2q}(q\p_1+X_{0,-2q})$.
Adjusting the values of $\bl_{0,-2q}$ and $a_{0,-2q}$, we may assume that $\bl_{0,-2q}=\l_2^{-2q}$ and $a_{0,-2q}=1$.
Now we conclude that $f_\bm(X_\bm,\p_1)=\bl^\bm q(\p_1-m_1\a)+\bl^\bm X_\bm$ for $\bm\in\Z^2$.
This complete the proof this claim and the lemma. 
\end{proof}

By Lemma, we see that the action of $\BB(q)$ on the module $M$ can be written as 
$L_{\bm}\cdot 1=\bl^\bm(q\p_1-m_1\a)+\bl^\bm X_\bm$. So we have the following description.

\begin{theorem}\label{char}
Any $U(\fh)$-free module of rank $1$ over $\BB(q)$ or $\BB'(q)$ 
is isomorphic to some $\Omega(\bl,\a)$ for some $\bl=(\l_1,\l_2)\in\C^2$ with $\l_1,\l_2\neq0$ and 
$\a\in\C$.
\end{theorem}

\vskip1cm

\noindent X.G.: School of Mathematics and Statistics, Zhengzhou University,
Zhengzhou 450001, Henan, P. R. China. Email: guoxq@zzu.edu.cn

\noindent X.L.: School of Mathematics and Statistics, Zhengzhou University,
Zhengzhou 450001, Henan, P. R. China. Email: liuxw@zzu.edu.cn

\end{document}